\documentclass[11pt,a4paper,reqno]{amsart}

\usepackage{epsfig}
\usepackage{amsfonts}
\usepackage{amsmath}
\usepackage{amssymb}
\usepackage{amsthm}
\usepackage{times}
\usepackage{color}
\usepackage{enumerate}
\usepackage[T1]{fontenc}
\usepackage[latin1]{inputenc}
\usepackage{graphicx}
\usepackage{nicefrac}
\usepackage[includehead,includefoot,margin=16mm]{geometry}
\usepackage[all]{xy}

\newtheorem{theorem}{Theorem}[section]
\newtheorem*{theorem*}{Theorem}
\newtheorem{corollary}[theorem]{Corollary}
\newtheorem{lemma}[theorem]{Lemma}
\newtheorem{proposition}[theorem]{Proposition}

{\theoremstyle{remark}
  \newtheorem{remark}[theorem]{Remark}}
{\theoremstyle{definition}
  \newtheorem{definition}[theorem]{Definition}

  \newtheorem{example}[theorem]{Example}

}

\def\Z{\mathbb{Z}}

\def\C{\mathbb{C}}
\def\CT{\mathbb{C}^*}

\def\T{\mathcal{T}}
\def\M{\mathcal{M}}
\def\N{\mathcal{N}}
\def\S{\mathcal{S}}
\def\V{\mathcal{V}}

\newcommand{\CC}[0]{\ensuremath{\mathbb{C}}}
\newcommand{\ZZ}[0]{\ensuremath{\mathbb{Z}}}

\newcommand{\RR}[0]{\ensuremath{\mathbb{R}}}

\newcommand{\QQ}[0]{\ensuremath{\mathbb{Q}}}

\newcommand{\lC}[0]{\ensuremath{\mathbb{C}}}

\newcommand{\ra}{\rangle}
\newcommand{\la}{\langle}

\newcommand{\spec}[0]{\ensuremath{\operatorname{Spec}}}

\newcommand{\li}{\varprojlim}
\newcommand{\ld}{\varinjlim}

\makeatletter \newcommand*\bigcdot{\mathpalette\bigcdot@{.5}} \newcommand*\bigcdot@[2]{\mathbin{\vcenter{\hbox{\scalebox{#2}{$\m@th#1\bullet$}}}}} \makeatother

\begin{document}

\title{On toric ind-varieties and pro-affine semigroups}

\author{Roberto D\'iaz}
\address{Instituto de Matem\'atica y F\'\i sica, Universidad de Talca,
  Casilla 721, Talca, Chile.}%
\email{robediaz@utalca.cl}

\author{Alvaro Liendo} %
\address{Instituto de Matem\'atica y F\'\i sica, Universidad de Talca,
  Casilla 721, Talca, Chile.}%
\email{aliendo@inst-mat.utalca.cl}

\date{\today}

\thanks{
 \mbox{\hspace{11pt}}{\it Key words}: ind-varieties, toric varieties, filtered semigroups, inductive and projective limits.\\
 \mbox{\hspace{11pt}}Both authors were partially supported by the grant 346300 for IMPAN from the Simons Foundation and the matching 2015-2019 Polish MNiSW fund. The first author was also partially supported by CONICYT-PFCHA/Doctorado Nacional/2016-folio 21161165. The second author was partially supported by Fondecyt project 1200502.}

\begin{abstract}
   An ind-variety is an inductive limit of closed embeddings of algebraic varieties and an ind-group is a group object in the category of ind-varieties. These notions were first introduced by Shafarevich in the study of the automorphism group of affine spaces and have been studied by many authors afterwards. An ind-torus is an ind-group obtained as an inductive limit of closed embeddings of algebraic tori that are also algebraic group homomorphisms. In this paper, we introduce the natural definition of toric ind-varieties as ind-varieties having an ind-torus as an open set and such that the action of the ind-torus on itself by translations extends to a regular action on the whole ind-variety. We are brought to introduce and study pro-affine semigroup that turn out to be unital semigroups isomorphic to closed subsemigroups of the group of arbitrary integer sequences with the product topology such that their projection to the first $i$-th coordinates is finitely generated for all positive integers $i$. Our main result is a duality between the categories of affine toric ind-varieties and the the category of pro-affine semigroups. 
\end{abstract}

\maketitle


\section*{Introduction}

Shafarevich first introduced in \cite{S66,S81} the notion of
infinite-dimensional algebraic varieties and infinite-dimen\-sional
algebraic groups, the so called ind-varieties and ind-groups,
respectively. These notions were later expanded and revisited by
several authors, see for instance \cite{Kum02,Kam03,S12} and the
recent preprint \cite{FK} that includes a detailed exposition of
generalities on ind-varieties and ind-groups.  In this paper, we
generalize the notion of toric varieties to the category of
ind-varieties.

We work over the field of complex numbers $\C$. An \emph{ind-variety} is a set $\mathcal{V}$ together with a filtration $V_1\hookrightarrow V_2\hookrightarrow\dots$ such that $\mathcal{V}=\bigcup V_i$, where each $V_i$ is a finite-dimensional algebraic variety and the inclusions
$\varphi_i\colon V_i{\hookrightarrow} V_{i+1}$ are closed embeddings. Morphisms in the category of ind-varieties are defined in the natural way, see Section~\ref{sec:ind-def} for details. An \emph{ind-group} is a group object in the category of ind-varieties, i.e., it is an ind-variety endowed with a group structure such that the inversion and multiplication maps are morphisms of ind-varieties. The set
\[\left(\lC^*\right)^{\infty}=\left\{\left(a_1,a_2,\dots\right)\mid a_i\in \lC^* \ \text{and}\ a_i\neq1\mbox{ for finitely many } i\right\}\] %
with the canonical structure of ind-variety given by the filtration $\lC^*\stackrel{\varphi_1}{\hookrightarrow}\left(\lC^*\right)^{2}\stackrel{\varphi_2}
{\hookrightarrow}\dots$, 
where $\varphi_i\left(a_1,\dots ,a_i\right)=\left(a_1,\dots ,a_i,1\right)$ for all integer $i>0$, has a natural structure of ind-group where the group law is given by component-wise multiplication. An \emph{algebraic torus} $T$ is an algebraic group isomorphic to $\left(\lC^*\right)^{k}$ for some integer $k\geq 0$. An \emph{ind-torus} $\T$ is an ind-group isomorphic to either an algebraic torus or  $\left(\lC^*\right)^{\infty}$. 

A toric variety $V$ is an irreducible algebraic variety having an
algebraic torus $T$ as an open set and such that the action of $T$ on
itself by translations extends to a regular action on $V$. Toric
varieties can be classified by certain combinatorial devices, see
\cite{T88,F93,CLS11}. This classification allows to translate many
algebro-geometric properties of a toric variety in combinatorial terms
that may then be computed algorithmically. Hence, toric varieties
represent a fertile testing ground for theories in algebraic
geometry. Toric morphisms between toric varieties are characterized by
the property that they restrict to a morphism of algebraic groups
between the corresponding algebraic tori. For affine toric varieties
their combinatorial nature is represented by the fact that the
category of affine toric varieties is dual to the category of affine
semigroups, i.e., finitely generated  semigroups that can be embedded
in $\Z^k$ for some integer $k\geq 0$. By convention, all our
semigroups will be commutative and unital. A unital semigroup is
usually called a monoid. 

In this paper we introduce the natural notion of toric ind-variety. A \emph{toric ind-variety} $\V$ is an ind-variety having an ind-torus $\T$ as an open set and such that the action of $\T$ on itself by translations extends to a regular action on $\V$, see Definition~\ref{def:ind-toric-variety}. Furthermore, toric morphisms between toric ind-varieties are morphisms that restrict to morphisms of ind-groups between the corresponding ind-tori, see Definition~\ref{def:ind-toric-morphism}. 
Our first result in this paper, contained in Theorem~\ref{toric-filtration}, shows that every toric ind-variety can be obtained as an inductive limit of toric varieties. This result allows us to investigate toric ind-varieties applying usual methods from toric geometry.

In Section~\ref{sec:pro-affine-semi} we introduce the natural dual objects to affine toric ind-varieties that we call pro-affine semigroups. We need to develop the theory of pro-affine semigroups from scratch since, up to our knowledge, only the case of pro-finite semigroups has been previously studied in the literature in detail, see for instance \cite{CH83}.  Let $\S$ be a commutative unital semigroup. In analogy with the case of topological algebras  \cite[Section~9.2]{N68} taking into account the lack of the notion of ideal of a semigroup, the natural way to endow the semigroup $\S$ with a topology is with a descending filtration $R_1\supset R_2\supset\dots$ of $\S\times\S$ of equivalence relations on $\S$ that satisfy certain compatibility condition with respect to the semigroup operation allowing to define a semigroup operation in the set of equivalence classes $\S/R_i$, see Section~\ref{sec:pro-affine-semi} for details. We call a semigroup $\S$ endowed with such a filtration a \emph{filtered semigroup}. A \emph{pro-affine semigroup} $\mathcal{S}$ is a filtered semigroup with filtration $R_1\supset R_2\supset\dots$ of compatible equivalence relations in $\S$ that is complete and such that $\S/R_i$ is an affine semigroup, for all integer $i>0$. Our main result concerning pro-affine semigroups is contained in Corollary~\ref{other-def-pro-affine} and is a classification of pro-affine semigroups as semigroups isomorphic to subsemigroups $\S$ of $\Z^\omega$, the group of arbitrary sequences of integers, that are closed in the product topology and such that $\pi_i(\S)$ is finitely generated for all integer $i>0$, where $\pi_i\colon \Z^\omega\rightarrow \Z^i$ is the projection to the first $i$-th coordinates. 

Finally, our main result in this paper is Theorem~\ref{main-theorem} where we show that the category of affine toric ind-varieties with toric morphisms is dual to the category of pro-affine semigroups with homomorphisms of semigroups. 

\smallskip

The contents of the paper is as follows. In Section~\ref{sec:ind-toric} we collect the preliminary notions of toric varieties, inductive and projective limits and ind-varieties required in this paper. In Section~\ref{sec:toric-ind} we introduce toric ind-varieties. In Section~\ref{sec:pro-affine-semi} we define pro-affine semigroups. In Section~\ref{sec:dual-category} we prove the duality of categories that is our main result. Finally, in Section~\ref{sec:examples} we provide some examples to ilustrate our results.

\subsection*{Acknowledgements}  

The authors would like to thank the anonymous referee of this manuscript for useful comments and for spotting a gap in a proof. Part of this work was done during a stay of both authors at IMPAN in Warsaw. We would like to thank IMPAN and the organizers of the Simons semester ``Varieties: Arithmetic and Transformations'' for the hospitality.

\section{Preliminaries}
\label{sec:ind-toric}
 
In this section we recall the notions of toric geometry, injective and projective limits and ind-varieties needed for this paper.

\subsection{Toric varieties} \label{sec:toric-varieties}

To fix notation we recall the basics of toric geometry. For details, see \cite{T88,F93,CLS11}. An algebraic torus $T$ is a linear algebraic group isomorphic to $(\C^*)^k$ for some integer $k\geq 0$.  A toric variety on $\C$ is an irreducible algebraic variety $V$ having an algebraic torus as a dense open set such that the action of $T$ on itself by translations extends to a regular action of $T$ on $V$. Similarly to \cite{CLS11}, we will not assume that a toric variety is necessarily normal. It is well known that affine toric varieties are in correspondence with affine semigroups $S$, i.e., with finitely generated semigroups that admit an embedding in $\Z^k$ for some integer $k\geq 0$. By convention, all our semigroups are commutative and unital.

Indeed, given an affine semigroup $S$, the corresponding affine toric
variety is given by $\V(S)=\spec\C[S]$, where $\C[S]$ is the semigroup
algebra given by $\C[S]=\bigoplus_{m\in S} \C\cdot\chi^m$. Here,
$\chi^m$ are new symbols and the multiplication rule is defined by
$\chi^0=1$ and $\chi^{m}\cdot\chi^{m'}=\chi^{m+m'}$. On the other
hand, the character lattice $M$ of the torus $T$ is a finitely
generated free abelian group $M\simeq \Z^k$ of rank $k=\dim T$. Let
$V$ be an affine toric variety with acting torus $T$. We define the
semigroup $\S(V)$ of the toric variety $V$ as the semigroup of
characters of $T$ in $M$ that extend to regular functions on $V$.

A toric morphism between toric varieties is a regular map that restricts to a morphism of algebraic groups between the corresponding algebraic tori acting on each toric variety. It is well known that the assignments $\V(\bigcdot)$ and $\S(\bigcdot)$ extend to functors from the category of affine varieties with toric morphisms to the category of affine semigroups and vice versa, respectively. Furthermore, the functors $\V(\bigcdot)$ and $\S(\bigcdot)$ together form a duality between the categories of affine toric varieties with toric morphisms and affine semigroups with homomorphisms of semigroups.

\subsection{Inductive and projective limits}

In this paper we will require several instances of inductive and projective limits of algebraic and geometric objects. We give here a brief account to fix notation, for details, see any reference on category theory such as \cite[Chapter~III]{S71}. All the systems of morphism required in this paper will be indexed by the positive integers with the usual order. Hence we restrict the exposition to this setting.

An inductive system indexed by the positive integers in a category $\mathcal{C}$ is a sequence
$$X_1\stackrel{^{\varphi_{1}}}{\rightarrow} X_2\stackrel{^{\varphi_{2}}}{\rightarrow} X_3\stackrel{^{\varphi_{3}}}{\rightarrow}\dots\,,$$
where $X_i$ are objects in $\mathcal{C}$ and $\varphi_i\colon X_i\rightarrow X_{i+1}$ are morphisms in $\mathcal{C}$. We denote such an inductive system by $(X_i,\varphi_i)$. For every $i,j>0$ with $i\leq j$, we define $\varphi_{ij}\colon X_i\rightarrow X_j$ as $\varphi_{ij}=\varphi_j\circ\varphi_{j-1}\circ\cdots\circ\varphi_i$, where by definition $\varphi_{ii}=\operatorname{id}\colon X_i\rightarrow X_i$. The inductive limit of an inductive system $(X_i,\varphi_i)$ is an object $\ld X_i$ in $\mathcal{C}$ and morphisms $\psi_i\colon X_i\rightarrow \ld X_i$ verifying $\psi_i=\psi_j\circ\varphi_{ij}$ and satisfying the following universal property: if there exist another object $Y$ and morphisms $\psi'_i\colon X_i\rightarrow Y$ verifying $\psi'_i=\psi'_j\circ\varphi_{ij}$, then there exist a unique morphism $u\colon\ld X_i\rightarrow Y$ such that $\psi'_i=u\circ \psi_i$ for all $i>0$.

The notion of projective limit is dual to the notion of inductive
limit and is defined as follows. A projective system indexed by the
positive integers in a category $\mathcal{C}$ is a sequence
$$X_1\stackrel{^{\varphi_{1}}}{\leftarrow} X_2\stackrel{^{\varphi_{2}}}{\leftarrow} X_3\stackrel{^{\varphi_{3}}}{\leftarrow}\dots\,,$$
where $X_i$ are objects in $\mathcal{C}$ and $\varphi_i\colon X_{i+1}\rightarrow X_i$ are morphisms in $\mathcal{C}$. We denote such a projective system by $(X_i,\varphi_i)$. For every $i,j>0$ with $i\leq j$, we define $\varphi_{ij}\colon X_j\rightarrow X_i$ as $\varphi_{ij}=\varphi_i\circ\varphi_{i+1}\circ\cdots\circ\varphi_j$, where by definition $\varphi_{ii}=\operatorname{id}\colon X_i\rightarrow X_i$. The projective limit of a projective system $(X_i,\varphi_i)$ is an objet $\li X_i$ in $\mathcal{C}$ and morphisms $\pi_i\colon \li X_i \rightarrow X_i$ verifying $\pi_i=\varphi_{ij}\circ\pi_{j}$ and satisfying the following universal property: if there exist another object $Y$ and morphisms $\pi'_i\colon Y\rightarrow X_i$ verifying $\pi'_i=\varphi'_{ij}\circ\pi'_{j}$, then there exist a unique morphism $u\colon Y\rightarrow \li X_i$ such that $\pi'_i=\pi_i\circ u$ for all $i>0$.

Both limits may not exist in arbitrary categories but in the categories of our interest (sets, groups, rings, algebras, semigroups, topological space) both limits can be realized by explicit constructions. Indeed, the inductive limit $\ld X_i$ of an inductive system $ (X_i,\varphi_i)$ can be constructed as $\ld X_i= \bigsqcup_{i>0} X_i/\sim$, where $\sim$ is the equivalence relation given by $x_i\sim x_j$, where $x_i\in X_i$ and $x_j\in X_j$, if there exist $k$ verifying $i\leq k$ and $j\leq k$ such that $\varphi_{ik}(x_i)=\varphi_{jk}(x_j)$. The morphisms $\psi\colon X_i\rightarrow \ld X_i$ are induced by the natural injections $ X_i\rightarrow \bigsqcup_{i>0} X_i$. Furthermore, if the morphisms $\varphi_i$ are injective, then we can naturally regard each $X_i$ as a subobject of the inductive limit $\ld X_i$. On the other hand, the projective limit $\li X_i$ of the projective system $(X_i,\varphi_i)$ can be constructed as
$$\li X_i=\Big\{(x_1,x_2,\dots)\in \prod_{i>0} X_i\mid x_i\in
X_i\mbox{ and } \varphi_{ij}(x_j)=x_i\Big\}\,,$$ %
and the morphisms $\pi_i\colon \li X_i\rightarrow X_i$ are induced by the natural projections $\prod_{i>0} X_i\rightarrow X_i$. Furthermore, if the morphisms $\varphi_i$ are surjective, then we can naturally regard each $X_i$ as a quotient of the projective limit $\li X_i$. Finally, in the case where $X_i$ are topological spaces, the topology on the projective limit $\li X_i$ coincides with the subspace topology on $\prod_{i>0} X_i$ with the product topology.

\medskip

\begin{example} \label{split-system} %
  Two particular instances of the above construction will appear very often in this paper. Recall that $\Z^\omega$ is the group of arbitrary sequences of integer numbers. This group is also called the Baer-Specker group. A sequence in $a\in \Z^\omega$ is denoted by $a=(a_1,a_2,\dots)$. Equivalently, $\Z^\omega$ is the projective limit of the system $\Z^1\leftarrow \Z^2\leftarrow\dots$, where the morphisms $\varphi_i\colon \Z^{i+1}\rightarrow \Z^i$ are the projections forgetting the last coordinate. Furthermore, the subgroup of $\Z^\omega$ of eventually zero sequences is denoted by $\Z^\infty$, so $a\in \Z^\infty$ is such that $a_i=0$ except for finitely many positive integers $i$. Equivalently, $\Z^\infty$ is the inductive limit of the system $\Z^1\rightarrow \Z^2\rightarrow\dots$, where the maps are the injections setting the last coordinate to $0$.

  If we take any inductive or projective subsystem of the system defining $\Z^\infty$ or $\Z^\omega$, respectively with the obvious morphisms given by compositions, then the limits are canonically isomorphic to $\Z^\infty$ or $\Z^\omega$, respectively. More generally, a projective or inductive system is called split if every morphism in the system admits a section. It is a straightforward computation to show that for any split projective system $\Z^{n_1}\leftarrow \Z^{n_2}\leftarrow\dots$, with a strictly increasing sequence $n_1<n_2<\dots$ of positive integers, the limit is isomorphic to $\Z^\omega$. Similarly, for any split inductive system $\Z^{n_1}\rightarrow \Z^{n_2}\rightarrow\dots$, with a strictly increasing sequence $n_1<n_2<\dots$ of positive integers, the limit is isomorphic to $\Z^\infty$.
\end{example}

In the sequel we will need the following lemma showing that $\Z^\omega$ and $\Z^\infty$ are mutually dual. Showing that $\operatorname{Hom}(\Z^\infty,\Z)\simeq \Z^\omega$ is a straightforward exercise, but showing $\operatorname{Hom}(\Z^\omega,\Z)\simeq \Z^\infty$ is more involved, see \cite{E50} for the original proof or \cite[Example 3.22]{D19} for a modern proof.

\begin{lemma}\label{specker}
  The groups $\Z^\omega$ and $\Z^\infty$ are mutually dual and this duality is realized by the usual dot product
  \[\langle\ ,\ \rangle\colon \Z^\omega\times\Z^\infty \rightarrow \Z,\qquad (m,p)\mapsto \sum_{i>0}(m_i\cdot p_i)\,. \] %
\end{lemma}

\subsection{General ind-varieties}
\label{sec:ind-def}

In this section we introduce the necessary notions and results
regarding ind-varieties. The definitions are borrowed from
\cite{Kum02}, \cite{Kam03} and \cite{FK}.

Recall that an ind-variety is a set $\mathcal{V}$ together with a
filtration $V_1\hookrightarrow V_2\hookrightarrow\dots$ such that
$\mathcal{V}=\ld V_i:=\bigcup V_i$, each $V_i$ is a finite-dimensional
variety over $\lC$, and the inclusion
$\varphi_i\colon V_i{\hookrightarrow} V_{i+1}$ is a closed
embedding. An ind-variety $\mathcal{V}$ is affine if each $V_i$ is
affine. We also define the ind-topology on an ind-variety
$\mathcal{V}$ as the topology where a set $U\subset \mathcal{V}$ is
open if and only if $U\cap V_i$ is open in $V_i$ for all $i>0$. In
particular, the filtration
$V_1\hookrightarrow V_2\hookrightarrow\dots$ is an inductive system
and the set $\V$ is the inductive limit. The topology defined on $\V$
corresponds to the inductive topology given by this inductive
system. The dimension of $\mathcal{V}$ is $\lim \dim(V_i)$ as $i$
tends to infinity.

A morphism between ind-varieties $\mathcal{V}$ and $\mathcal{V'}$ with
filtrations $V_i$ and $V'_j$ respectively, is a map
$\varphi\colon\mathcal{V}\rightarrow \mathcal{V'}$ satisfying that for
every $i>0$ there exists a positive integer $j>0$ such that
$\varphi\left(V_i\right)\subset V'_{j}$ and
$\varphi|_{V_i}\colon V_i\rightarrow V'_{j}$ is a morphism of
varieties. A morphism $\varphi$ of ind-varieties is an isomorphism if
$\varphi$ is bijective and $\varphi^{-1}$ is a morphism of
ind-varieties. Furthermore, two filtrations
$V_1\hookrightarrow V_2\hookrightarrow\dots$ and
$W_1\hookrightarrow W_2\hookrightarrow\dots$ on the same underlying
set $\mathcal{V}$ are equivalent if the identity map is a isomorphism
of ind-varieties. In analogy with similar Example~\ref{split-system},
if we take any subfiltration of the filtration
$V_1\hookrightarrow V_2\hookrightarrow\dots$, the ind-varieties
obtained by both filtrations are isomorphic. The Cartesian product of two ind-varieties is again an ind-variety with the product filtration. Moreover, an ind-group is an
ind-variety $\mathcal{G}$ endowed with a group structure such that the
inversion and multiplication maps are morphisms of ind-varieties.

Recall that a topological space is irreducible if it is not equal to
the union of two proper closed sets. An irreducible ind-variety
$\mathcal{V}$ does not necessarily admit an equivalent filtration
$V_1\hookrightarrow V_2\hookrightarrow\dots$ by irreducible varieties
as shown in \cite[Remark~4.3]{BF13}, see also
\cite[Example~1.6.5]{FK}. An ind-variety $\mathcal{V}$ is called
curve-connected if for any two points $a,b\in \mathcal{V}$ there
exists an irreducible algebraic curve $C$ and a morphism
$C\rightarrow \mathcal{V}$ whose image contains $a$ and $b$. An
ind-variety $\mathcal{V}$ is curve-connected if and only if there
exists an equivalent filtration
$V_1\hookrightarrow V_2\hookrightarrow\dots$ by irreducible varieties
\cite[Proposition~1.6.3]{FK}.

Recall that a set in a topological space is locally closed if it is
the intersection of an open set and a closed set. Let
$\mathcal{V}=\ld{V_i}$ be an ind-variety. A subset
$A\subset \mathcal{V}$ is called algebraic if it is locally closed and
contained in $V_i$ for some $i>0$, so $A$ has a natural structure of
an algebraic variety.  A morphism
$\alpha\colon\mathcal{V}\rightarrow\mathcal{V}'$ is called an
embedding if the image $\alpha(\mathcal{V})\subset\mathcal{V}'$ is
locally closed and induces an isomorphism of ind-varieties between
$\mathcal{V}$ and $\alpha(\mathcal{V})$. An embedding is called a
closed embedding (resp. an open embedding) if
$\alpha(\mathcal{V})\subset\mathcal{V}'$ is closed
(resp. open). Finally, recall that a constructible set is a finite
union of locally closed subsets.

\begin{example} \label{15}
  \begin{enumerate}
  \item The infinite-dimensional vector space
    \[\C^{\infty}:=\{(a_1,\dots )\mid a_i\in\C\mbox{ and }
      a_i\neq0\mbox{ for finitely many }i\}\] %
    has a canonical structure of ind-variety given by the filtration $\C\stackrel{\varphi_1}{\hookrightarrow}\C^2\stackrel{\varphi_2}{\hookrightarrow}\C^3\stackrel{\varphi_3}{\hookrightarrow}\dots$ where $\varphi_n(a_1,\dots ,a_i)=(a_1,\dots ,a_i,0)$, for all $i>0$. This ind-variety is called the infinite-di\-men\-sion\-al affine space. Remark that we can change the complex number $0$ in the filtration definition of $\C^\infty$ and in $(i+1)$-th coordinate of $\varphi_i$ by any other number. The ind-variety obtained this way is easily seen to be isomorphic to $\C^\infty$. For instance, we denote by $\C^\infty_1$ the ind-variety isomorphic to the infinite-dimensional affine space given by $\C_1^{\infty}:=\{(a_1,\dots )\mid a_i\in\C\mbox{ and } a_i\neq 1\mbox{ for finitely many }i\}$.

  \item The set
    \[\left(\lC^*\right)^{\infty}=\left\{\left(a_1,a_2,\dots\right)\mid a_i\in \lC^* \ \text{and}\ a_i\neq1\mbox{ for finitely many }i\right\}\] %
    has a canonical structure of ind-variety given by the filtration $\lC^*\stackrel{\varphi_1} {\hookrightarrow}\left(\lC^*\right)^{2}\stackrel{\varphi_2} {\hookrightarrow}\left(\lC^*\right)^{3}\stackrel{\varphi_3} {\hookrightarrow}\dots$, where $\varphi_i\left(a_1,\dots ,a_i\right)=\left(a_1,\dots ,a_i,1\right)$ for all $i>0$. This ind-variety is an open set in the infinite-dimensional affine space. This follows straightforward from the isomorphism $\C^\infty\simeq \C^\infty_1$ above. Remark that $\left(\lC^*\right)^{\infty}$ has a natural structure of ind-group given by component-wise multiplication.
  \end{enumerate}
\end{example}

A commutative topological $\C$-algebra $\mathcal{A}$ is pro-affine if it is Hausdorff, complete and admits a base $\left\{I_i\right\}_{i>0}$ of open neighborhoods of $0$, where $I_i\subset \mathcal{A}$ is an ideal for all $i>0$. Furthermore, we can assume that $I_i$ form a descending filtration $I_1\supset I_2\supset\dots$ of ideals of $\mathcal{A}$. Recall that Hausdorff property is equivalent to $\bigcap I_i=\left\{0\right\}$ and completeness is equivalent to $\mathcal{A}=\li\mathcal{A}_i$ where the algebra $\mathcal{A}_i:=\mathcal{A}/ I_i$ is taken with the discrete topology, see \cite[Section~9.2]{N68} for details. A pro-affine algebra $\mathcal{A}$ is algebraic if $\mathcal{A}_i$ is finitely generated over $\C$ for all $i>0$. Every finitely generated algebra over $\C$ is a pro-affine algebraic with $I_i=\left\{0\right\}$ for all $i>0$. In the sequel all pro-affine algebras are algebraic, so we will drop algebraic from the notation.

For an ind-variety $\mathcal{V}$ with filtration $V_1\stackrel{\varphi_1}{\hookrightarrow} V_2\stackrel{\varphi_2}{\hookrightarrow},\dots$ the ring of regular functions $\lC[\mathcal{V}]$ is define as $\li\C[V_i]$ of $\lC[V_1]\stackrel{\varphi^{*}_1}{\leftarrow} \lC[V_2]\stackrel{\varphi^{*}_2}{\leftarrow}\dots$ where each $\lC[V_i]$ is taken with the discrete topology and $\li{\lC[V_i]}$ has the projective limit topology i.e.,
$$\lC[\mathcal{V}]=\li{\lC[V_i]}=\Big\{\left(f_1, f_2,\dots\right)\mid f_i\in
\C[V_i]\mbox{ and } \varphi^{*}_i\left(f_{i+1}\right)=f_i\Big\}\subset\prod_{i>0}\lC[V_i]\,,$$
with subspace topology. The projective limit comes equipped with natural projections $\pi_i\colon\C[\mathcal{V}]\rightarrow \C[V_i]$. 

Let  $\alpha\colon\mathcal{V}\rightarrow\V'$ be a morphism of ind-varieties, then for every $i>0$ there exists $j>0$ such that $\alpha$ induces an homomorphism $\CC[V'_j]\rightarrow \CC[V_i]$ and so $\alpha$ induces a continuous pro-affine algebras homomorphism $\alpha^*\colon\C[\V']\rightarrow\C[\mathcal{V}]$. Conversely, every continuous homomorphism $\beta\colon\CC[V']\rightarrow \CC[V]$ of pro-affine algebras induces for every $i>0$ a homomorphism $\CC[V'_j]\rightarrow \CC[V_i]$ for some $j>0$ and so it induces a morphism $V_i\rightarrow V_j$ which in turns gives a morphism $\beta^*\colon\mathcal{V}\rightarrow\mathcal{V}'$  \cite{Kum02,Kam03}. This yields an equivalence of categories between pro-affine algebras and affine ind-varieties.

\section{Toric ind-variety }
\label{sec:toric-ind}

An algebraic torus $T$ is an algebraic group isomorphic to $\left(\lC^*\right)^{i}$ for some $i\geq 0$. An ind-torus $\T$ is an ind-group isomorphic to either an algebraic torus or  $\left(\lC^*\right)^{\infty}$. A regular action of an ind-torus $\T$ on an ind-variety $\mathcal{V}$ is a group action $\alpha\colon\T\times \mathcal{V}\rightarrow \mathcal{V}$ by automorphisms of $\mathcal{V}$ such that $\alpha$ is also a morphism of ind-varieties.

\begin{definition} \label{def:ind-toric-variety} %
  A toric ind-variety is a curve-connected ind-variety $\mathcal{V}$
  having an ind-torus $\T$ as an open subset such that the action of
  $\T$ on itself by translations extends to a regular action of $\T$
  on $\mathcal{V}$.
\end{definition}

If $\mathcal{V}$ is finite dimensional, then this definition coincides
with the usual notion of toric variety since curve-connectedness is
equivalent to irreducibility in the finite-dimensional case, see for
instance \cite[Definition~1.1.3]{CLS11}. Asking for an ind-toric
variety $\mathcal{V}$ to be curve-connected is equivalent to ask that
$\mathcal{V}$ can be presented as the inductive limit of irreducible
varieties \cite[Proposition~1.6.3]{FK}.  Remark that similarly to
\cite{CLS11} and unlike other references \cite{T88,F93}, we do not
require toric varieties to be normal.

\begin{example} \label{different-torus} %
  Recall that $\Z^\infty$ is defined as the inductive limit of the
  inductive system $\Z\rightarrow \Z^2\rightarrow\dots$ where the maps
  are the injections setting the last coordinate to $0$. Taking tensor
  product of this system with $\C^*$ we obtain the inductive system
  defining $(\C^*)^\infty$. In analogy with the finite-dimensional
  case, we denote this by by
  $(\C^*)^\infty=\Z^\infty\otimes_{\Z}\C^*$.  Now, it follows directly
  from Example~\ref{split-system} that for every sequence
  $\lC^*\stackrel{\varphi_1}
  {\hookrightarrow}\left(\lC^*\right)^{2}\stackrel{\varphi_2}
  {\hookrightarrow}\left(\lC^*\right)^{3}\stackrel{\varphi_3}
  {\hookrightarrow}\dots$ with $\varphi_i$ injective homomorphisms of
  algebraic groups, the corresponding ind-variety is an ind-group
  isomorphic to $(\mathbb{C}^{*})^\infty$.
\end{example}

In the next theorem we show that for every toric ind-variety, we can find an equivalent filtration composed of toric varieties and toric morphisms.

\begin{theorem}\label{toric-filtration}
  Let $\mathcal{V}=\ld V_i$ be an ind-variety endowed with a regular
  action of the ind-torus $\T$. Then $\mathcal{V}$ is an affine toric
  ind-variety with respect to $\T$ if and only if
  $\mathcal{V}\simeq\ld{W_j}$ where $W_j$ are affine toric varieties with
  acting torus $T_j$, the closed embedding
  $\varphi_j\colon W_j\hookrightarrow W_{j+1}$ are toric morphisms and
  the ind-torus $\T$ is the inductive limit $\ld T_j$.
\end{theorem}

\begin{proof}
  The finite dimensional case is trivial since we can take
  $W_j=\mathcal{V}$ and $T_j=\T$, for all $j>0$. Hence, we only deal
  with the case where $\T=(\C^*)^{\infty}$. To prove the ``only if''
  part we may assume that each $V_i$ is irreducible since
  $\mathcal{V}$ is curve-connected. Let $W_j$ be the closure of
  $\left(\lC^*\right)^j$ in $\mathcal{V}$. The acting torus in $W_j$
  is $T_j=\left(\lC^*\right)^j$ and so it follows that $\T=\ld T_j$.
  Fix an integer $j>0$. Let $A$ be a closed set in $\mathcal{V}$. Then
  $A\cap (\CT)^\infty$ is closed in $(\CT)^\infty$ so
  $A\cap (\CT)^{j+1}$ is closed in $(\CT)^{j+1}$. Hence, the inclusion
  $(\CT)^{j+1}\hookrightarrow\mathcal{V}$ is continuous and so by
  \cite[Lemma~1.1.5]{FK}, there exist $i>0$ such that
  $W_{j+1}\subset V_i$. Furthermore, the inclusion
  $(\CT)^{j}\hookrightarrow (\CT)^{j+1}$ induces an inclusion
  $\varphi_j\colon W_j\hookrightarrow W_{j+1}$. Since $V_i$ is closed
  in $\mathcal{V}$ we have that $W_j$ and $W_{j+1}$ are closed in
  $V_i$ and so $\varphi_i$ is a closed embedding.

  We claim that the varieties $W_j$ are toric with respect to the
  algebraic tori $T_j=(\CT)^{j}$ and the morphisms
  $\varphi_j\colon W_j\hookrightarrow W_{j+1}$ are toric. Indeed,
  since $(\CT)^{j}$ is irreducible, $W_j$ is also irreducible, for all
  $j>0$. Furthermore, the $T_j$-action on $T_j$ by translations
  extends to a $T_j$-action in $W_j$ since for every $t\in T_j$, we
  have $t.W_j$ is contained in the closure of
  $t.(\C^*)^j= \left(\C^*\right)^j$ and so $W_j$ is stabilized by
  $T_j$. Finally, by \cite[Proposition 1.11]{B10}, the $T_j$-orbit
  $(\lC^*)^j$ is locally closed in $W_j$ and so we conclude that
  $(\lC^*)^j$ is an open set in $W_j$. Hence $W_j$ is a toric
  variety. Furthermore, the morphism
  $\varphi_j\colon W_j\hookrightarrow W_{j+1}$ is toric since its
  restriction to the acting torus is a group homomorphism by
  definition.

  Finally, we prove that $\mathcal{V}\simeq\ld{W_j}$ by proving that   the filtrations given by $V_i$ and $W_j$, respectively are   equivalent. We already proved above that for every $j>0$ there   exists $i>0$ such that $W_j\subset V_i$ is a closed embedding. To prove the other direction, we need to prove that for every $V_i$ there exists $W_k$ with $V_i\subset W_k$. Without loss of generality, we may and will assume that $W_1\subset V_i$. Observe that the set
  $X=V_i\cap \left(\C^*\right)^\infty$ is a non-empty algebraic subset of
  $\mathcal{V}$. Furthermore, since
  $\left(\C^*\right)^\infty\subset\bigcup_{j>0}W_j$ and
  $X\subset \left(\C^*\right)^\infty$ we have
  $X=\bigcup_{j>0}X\cap W_j$.  By \cite[Lemma~1.3.1]{FK}, there exists a positive integer $k$ such that $X=X\cap W_k$ and so
  $X\subset W_k$. Moreover, the closure of $X$ in $\mathcal{V}$ is $V_i$ since $V_i$ is irreducible by our assumption above. Since $W_k$ is closed, we conclude that $V_i\subset W_k$ is a closed embedding. This concludes the proof of the ``only if'' part of the theorem.

  We now prove the ``if'' direction of the theorem. The ind-variety
  $\mathcal{V}\simeq\ld W_j$ is curve-connected since each $W_j$ is
  irreducible.  Furthermore, by Example~\ref{different-torus} the
  limit $\T=\ld T_j$ is an ind-torus. Moreover, $\T$ is an open set in
  $\ld W_j$ by the definition of the ind-topology. Moreover, the
  action of $\T$ on itself by multiplication extends to $\ld W_j$
  since the same holds in all the strata for $T_j$ acting on
  $W_j$. This concludes the proof.
\end{proof}

\begin{remark}
  The above theorem can be generalized to the case of ind-varieties endowed with an action of a nested ind-group, i.e., an ind-group admitting an equivalent filtration by algebraic groups \cite[Section~9.4]{FK}. We restrict to the case of the ind-torus for simplicity.  
\end{remark}

Let $\mathcal{V}=\ld V_i$ be a toric ind-variety. We say that $V_1\hookrightarrow V_2\hookrightarrow\dots$ is a toric filtration if for every $i>0$ the variety $V_i$ is toric with acting torus $T_i$, the closed embedding $\varphi_i\colon V_i\hookrightarrow V_{i+1}$ is a toric morphism and the acting ind-torus $\T$ is the inductive limit $\ld T_i$. Theorem~\ref{toric-filtration} above ensures the every toric ind-variety admits a toric filtration.

We define toric morphisms in direct analogy with the case of classical toric varieties.

\begin{definition} \label{def:ind-toric-morphism}
  Let $\T_{\mathcal{V}}$ and $\T_{\mathcal{V'}}$ be ind-tori acting on toric ind-varieties $\mathcal{V}=\ld V_i$ and $\mathcal{V'}=\ld V'_j$, respectively. A morphism $\alpha\colon\mathcal{V}\rightarrow \mathcal{V}'$ of ind-varieties is toric if the image of $\T_{\mathcal{V}}$ by $\alpha$ is contained in $\T_{\mathcal{V}'}$ and $\alpha|_{\T_\mathcal{V}}\colon\T_\mathcal{V}\rightarrow \T_{\mathcal{V}'}$ is a morphism of ind-group.
\end{definition}

\begin{proposition} \label{toric-morphism-filtration} %
  Let $\alpha\colon\V\rightarrow \V'$ be a morphism of toric ind-varieties $\mathcal{V}$ and $\V'$. Then $\alpha$ is a toric morphism if and only if for every toric filtrations $V_1\hookrightarrow V_2\hookrightarrow\dots$ and $V'_1\hookrightarrow V'_2\hookrightarrow\dots$ of $\V$ and $\V'$, respectively, and every $i>0$, there exists an integer $j>0$ such that $\alpha|_{V_i}\colon V_i\rightarrow V'_j$ is a toric morphism.
\end{proposition}

\begin{proof}
  To prove the ``only if'' direction of the proposition, we assume that $\alpha$ is toric and by Theorem~\ref{toric-filtration} we let $V_1\hookrightarrow V_2\hookrightarrow\dots$ and $V'_1\hookrightarrow V'_2\hookrightarrow\dots$ be toric filtrations of $\V$ and $\V'$, respectively. By definition of morphism of ind-varieties, for every $i>0$ there exists $j>0$ such that $\alpha$ restricts to a morphism of varieties $\alpha|_{V_i}\colon V_i\rightarrow V'_j$. Let $\T_{\V}=\ld T_i$ and $\T_{\V'}=\ld H_j$ be the acting tori with the filtration coming from the toric filtration of $\V$ and $\V'$, respectively. By the definition of toric morphism, we have $\alpha(T_i)\subset \T_{\V'}$ and so $\alpha(T_i)\subset H_j=V'_j\cap \T_{\V'}$. Since $\alpha\colon\T_{\V}\rightarrow \T_{\V'}$ is a group homomorphism, the same holds for $\alpha|_{T_i}\colon T_i\rightarrow H_j$. This proves this direction of the proposition.

  To prove the ``if'' part, we let $V_1\hookrightarrow V_2\hookrightarrow\dots$ and $V'_1\hookrightarrow V'_2\hookrightarrow\dots$ be toric filtrations of $\V$ and $\V'$, respectively. We further assume that for every $i>0$, there exist an integer $j>0$ such that $\alpha|_{V_i}\colon V_i\rightarrow V'_j$ is a toric morphism. Furthermore, replacing the toric filtration of $\V'$ by a renumbered subfiltration we may and will assume $\alpha|_{V_i}\colon V_i\rightarrow V'_i$ is a toric morphism. It follows that $\alpha(T_i)\subset H_i$, where $\T_{\V}=\ld T_i$ and $\T_{\V'}=\ld H_j$ be the acting tori with the filtration coming from the toric filtration of $\V$ and $\V'$, respectively. Hence, we conclude $\alpha(\T_{\V})\subset \T_{\V'}$. Similarly, the fact that $\alpha|_{T_i}\colon T_i\rightarrow H_i$ is a homomorphism of groups implies that $\alpha|_{\T_{\V}}\colon \T_{\V}\rightarrow \T_{\V'}$ is a homomorphism of ind-groups proving the proposition.
\end{proof}

\begin{remark}
  It is straightforward to show that a toric morphism $\alpha\colon \V\rightarrow \V'$ of toric ind-varieties is equivariant, i.e., $\alpha(t.x)=\alpha(t).\alpha(x)$, for all $t\in \T_{\V}$ and all $x\in \mathcal{V}$.
\end{remark}

A character of an ind-torus $\T$ is a morphism $\chi\colon\T\rightarrow \C^*$ of ind-varieties that is also a group homomorphism. The set of characters of $\T$ forms a group denoted by $\mathcal{M}$. If $\dim\T< \infty$ it is well known that $\mathcal{M}$ is a finitely generated free abelian group of rank $\dim\T$. Similarly, a one-parameter subgroup of $\T$ is a morphism $\lambda\colon\C^*\rightarrow \T$ of ind-varieties that is also a group homomorphism. The set of one-parameter subgroups of $\T$ forms a group denoted by $\mathcal{N}$. If $\dim\T< \infty$ it is well known that $\mathcal{N}$ is also a finitely generated free abelian group of rank $\dim\T$. Furthermore, if $\dim\T< \infty$, then the groups $\mathcal{M}$ and $\mathcal{N}$ are dual with duality $\M\times\N\rightarrow \Z$ given by $\langle\chi,\lambda\rangle=k$ where $k$ is the unique integer such that $\chi\circ\lambda\colon \C\rightarrow \C$ maps $t$ to $t^k$.

We now compute the groups of characters and one-parameter subgroups of the ind-torus and prove the analogous duality result. Let $\T$ be the infinite-dimensional ind-torus with toric filtration $T_1\hookrightarrow T_2\hookrightarrow\dots$. Letting $M_i$ and $N_i$ be the character lattice and the one-parameter subgroup lattice of $T_i$, respectively, the filtration induces naturally a projective system $M_1\leftarrow M_2\leftarrow\dots$ and an inductive system $N_1\rightarrow N_2\rightarrow\dots$.

\begin{proposition} \label{ind-torus-characters} Let $\T$ be the infinite-dimensional ind-torus with toric filtration $T_1\hookrightarrow T_2\hookrightarrow\dots$. Then
  \begin{enumerate}
  \item The group of characters $\mathcal{M}$ of $\T$ is $\li M_i$ and is isomorphic to $\Z^\omega$.
  \item The group of one-parameter subgroups $\mathcal{N}$ of $\T$ is $\ld N_i$ and is isomorphic to $\Z^\infty$.
  \item The groups $\mathcal{M}$ and $\mathcal{N}$ are natural dual to each other and the duality is realized by the pairing $\la\ ,\ \ra\colon \mathcal{M}\times \mathcal{N}\rightarrow \Z$ given by $\langle\chi,\lambda\rangle=k$, where $\lambda\circ \chi\colon\C^*\rightarrow \C^*$ maps $t\mapsto t^k$ making $\mathcal{M}$ and $\mathcal{N}$ dual groups.
  \end{enumerate}
\end{proposition}

\begin{proof}
  To prove (1), we let $\chi\colon\T\rightarrow \C^*$ be a character of $\T$. By the definition of morphism of ind-varieties, we have that $\chi|_{T_i}\colon T_i\rightarrow \C^*$ is a character of $T_i$ for all $i>0$. This produces homomorphisms $\pi_i\colon\M\rightarrow M_i$ satisfying $\pi_i=\varphi^*_i\circ\pi_{i+1}$, where $\varphi^*_i\colon M_{i+1}\rightarrow M_i$ is the map induced by $\varphi\colon T_i\rightarrow T_{i+1}$. By the universal property of the projective limit we have a homomorphism $\M\rightarrow \li M_i$. On the other hand, we define the inverse homomorphism $\li M_i\rightarrow \M$ in the following way. Let $(\chi_1,\chi_2,\dots)$ be an element in the projective limit $\li M_i$. We associate a character $\chi\in\M$ given by $\chi\colon\T\rightarrow\C^*$ via $t\mapsto \chi_k(t)$ for any $k>0$ such that $t\in T_k$. By the definition of projective limit this map is well defined. It is a straightforward verification that it is a homomorphism. This proves that $\M$ is the projective limit $\li M_i$. Finally, $\M$ is isomorphic to $\Z^\omega$ by Example~\ref{split-system}.

  To prove (2), let $\lambda_i\colon \C^*\rightarrow T_i$ be a one-parameter subgroup in $N_i$. Composing with the injection $T_i\hookrightarrow \T$ we obtain a one-parameter subgroup $\lambda\colon \C^*\rightarrow \T$ of the ind-torus. This yields homomorphisms $\psi_i\colon N_i\rightarrow \N$. By the universal property of the inductive limit we have a homomorphism $\ld N_i\rightarrow \N$. On the other hand, we define the inverse homomorphism in the following way. Let $\lambda\colon\C^*\rightarrow \T$ be a one-parameter subgroup of $\T$. By the definition of morphism of ind-varieties, we have that there exists $k>0$ such that the one-parameter subgroup $\lambda$ restricts to $\lambda_k\colon \C^*\rightarrow T_k$ is a one-parameter subgroup of $T_k$. Hence, $\lambda_k\in N_k$ and composing with $\psi_k\colon N_k\rightarrow \ld N_i$ we obtain a homomorphism $\N\rightarrow \ld N_i$. By the definition of inductive limit this map is well defined. It is a straightforward verification that it is a homomorphism. Finally, $\N$ is isomorphic to $\Z^\infty$ by Example~\ref{split-system}.

  To prove (3), a routine computation shows that $\langle\ ,\ \rangle$ is bilinear and under the isomorphisms in (1) and (2) corresponds to the usual dot product defined in Lemma~\ref{specker}. This proves the proposition.
\end{proof}

In the proof of our main result, we will need the following lemma whose proof is straightforward.

\begin{lemma}\label{morph-tori}
  Let $\T$ and $\T'$ be ind-tori and let $\alpha\colon\T\rightarrow\T'$ be a ind-group homomorphism with character group $\M_{\T}$ and $\M_{\T'}$ and one-parameter subgroup group $\N_{\T}$ and $\N_{\T'}$. Then $\alpha$ induces homomorphisms $\alpha^*\colon\M_{\T'}\rightarrow \M_{\T}$ and $\alpha_*\colon \N_{\T}\rightarrow \N_{\T'}$.
\end{lemma}

\section{Pro-affine semigroup}
\label{sec:pro-affine-semi}

A semigroup is a set $(\S, +)$ with an associative binary operation. All our semigroups will be  commutative and unital. A semigroup $S$ is called affine if it is finitely generated and can be embedded in a $\Z^k$ for some $k\geq 0$. It is well known that the category of affine toric varieties with toric morphisms is dual to the category of affine semigroups with homomorphisms of semigroups. The main result of this paper is a generalization of this result to the case of affine toric ind-varieties. In this section, we define and study the semigroups $\mathcal{S}$ that will appear as the semigroup of an affine toric ind-variety $\mathcal{V}$.

Recall that the ring of regular functions $\C[\mathcal{V}]$ of and ind-variety is a pro-affine algebra and so it is endowed with a topology holding the information of the filtration of $\mathcal{V}$ \cite{Kam03}. We will first transport the notion of pro-affine algebra into the context of semigroups. A pro-affine algebra $\mathcal{A}$ is defined using a filtration of ideals on $\mathcal{A}$ and the projective limit topology induced by the quotients of $\mathcal{A}$ by the ideals in this filtration. In the case of semigroups, there exits an analog notion of ideal, but there is no bijection between ideals and quotient semigroups. For this reason, in the context of semigroups, we need the more general notion of compatible equivalence relations to keep track of all the possible quotients.

\smallskip

An equivalence relation on a set $\S$ is a subset $R\subset \S\times \S$ satisfying the usual properties of being reflexive, symmetric and transitive. An equivalence relation on a semigroup $\S$ is called compatible if for every $(m,n)$ and $(m',n')$ in $R$ we have that $(m+m',n+n')$ also belongs to $R$. In this case, the set of equivalence classes $\S/R$ inherits a natural structure of semigroup with binary operation given by $[m]+[m']=[m+m']$, where $[m]$ denotes the class of $m$ in $\S/R$.

A filtered semigroup is a couple $(\S,F)$, where $\S$ is a semigroup and $F$ is a descending filtration $R_1\supset R_2\supset\dots$ of $\S\times\S$ of compatible equivalence relations on $\S$. We denote a filtered semigroup simply by $\S$ if $F$ is clear from the context. In close analogy with \cite[Section~9.2]{N68}, the filtration of compatible equivalence relations on $\S$ defines a topology on $\S$ having basis $\{E_{m,k}\mid m\in \S, k>0\}$, where $E_{m,k}=\left\{m'\in \mathcal{S}\mid (m,m')\in R_k\right\}$ is the equivalence class of $m$ under the equivalence relation $R_k$. It is straightforward to verify that this topology coincides with the finest topology making all the quotient morphisms $\S\rightarrow \S/R_k$ continuous where $\S/R_k$ is taken with the discrete topology. The trivial equivalence relation on $\S$ corresponds to the diagonal in $\S\times\S$. The trivial filtration on a semigroup $\S$ is given by setting each equivalence relation $R_i$ to be trivial. In this case the induced topology on $\S$ is the discrete topology.

Let $\mathcal{S}$ be filtered semigroup with filtration $R_1\supset R_2\supset\dots$ of compatible equivalence relations in $\S$.  It is straightforward to verify that the topology on $\mathcal{S}$ is Hausdorff if and only if $\bigcap_{k>0} R_k$ equals the diagonal in $\S\times\S$. Additionally, we can generalize the notion of Cauchy sequence to this context of semigroups. Indeed, a sequence $\left\{a^{(i)}\right\}_{i>0}\subset\mathcal{S}$ in the semigroup is say to be Cauchy sequence if given any $k>0$ there exists an integer $N$ such that $\left(a^{(i)},a^{(j)}\right)\in R_k$ for all $i,j>N$. A direct computation shows that a convergent sequence is always Cauchy. We say that a filtered semigroup $\mathcal{S}$ is complete if every Cauchy sequence converges.

Given a projective system $S_1\leftarrow S_2\leftarrow\dots$ of semigroups we define a filtration $R_1\supset R_2\supset\dots$ of compatible equivalence relations on the projective limit $\S=\li S_i$ by $R_i=\{(m,m')\in \S\times\S\mid \pi_i(m)=\pi_i(m')\}$. The topology induced on $\S$ by this filtration coincides with the projective limit topology.

\begin{proposition}\label{proj-complete}
  Let $S_1\leftarrow S_2\leftarrow\dots$ be a projective system of semigroups where each $S_i$ carries the discrete topology. Then the projective limit semigroup $\S=\li S_i$ is Hausdorff and complete.
\end{proposition}

\begin{proof}
  A couple $(m,m')\subset \S\times\S$ belongs to $R_k$ if and only if $m_i=m_i'$ for all $i\leq k$. Hence, the couple $(m,m')$ belongs to $\bigcap_{k>0}R_k$ if and only if $m=m'$. We conclude that $\bigcap_{k>0}R_k$ equals the diagonal of $\S\times\S$ and so $\S$ is Hausdorff. To prove that $\S$ is complete, let $\left\{m^{(i)}\right\}_{i>0}\subset\S$ be a Cauchy sequence in $\S$. Recall that, by the definition of projective limit, each $m^{(i)}$ equals $\left(m^{(i)}_1,m^{(i)}_2,\dots\right)\in \prod_{i>0}S_i$. For every $k>0$ there exist $N$ such that $\left(m^{(i)},m^{(i+1)}\right)\in R_k$ for all $i> N$. Hence, for every $k$ there exist $N$ such that $m^{(i)}_k=m^{(i+1)}_k=m^{(i+2)}_k=\cdots$ when $i> N$. Letting $m_k=m^{(i)}_k\in S_k$ for any $i>N$, we let $m=(m_1,m_2\dots )\in \S$. Now, for every $k>0$ there exist $N$ such that $\left(m,m^{(i)}\right)\in R_k$ for all $i> N$ and so the Cauchy sequence $\left\{m^{(i)}\right\}_{i>0}\subset\S$ converges to $m$.
\end{proof}

\begin{remark} \label{limit-equal-complete} %
  If a filtered semigroup $\mathcal{S}$ with filtration $R_1\supset R_2\supset\dots$ of compatible equivalence relations in $\S$ is Hausdorff and complete, then $\li S_i$, where $S_i=\S/R_i$ with the morphism induced from $R_i\supset R_{i+1}$, is canonically isomorphic to $\S$. Indeed, the canonical map $\S\rightarrow \li S_i$ into the projective limit given by $m\mapsto (\pi_1(m),\pi_2(m),\dots)$ has inverse given by $([m_1],[m_2],\dots)\mapsto \lim m_i$, where $\{m_i\}_{i>0}$ is the Cauchy sequence given in $\S$ by $\{m_1,m_2,\dots\}$.
\end{remark}

We now define the natural notion of morphism of filtered semigroups.

\begin{definition} \label{def-morph-semi} %
  Let $S$ and $S'$ be filtered semigroups with filtrations $R_1\supset R_2\supset\dots$ and $R'_1\supset R'_2\supset\dots$, respectively.  A map $\beta\colon\S\rightarrow \S'$ is called a morphism of filtered semigroups if $\beta$ is a semigroup homomorphism and for every $i>0$ there exists $j>0$ such that $(\beta\times\beta)(R_j)\subset R'_i$.  In particular, every morphism $\beta\colon\S\rightarrow \S'$ of filtered semigroups is continuous since the condition $(\beta\times\beta)(R_j)\subset R'_i$ implies point-wise continuity at every $m\in \S$. As usual, an isomorphism $\beta\colon\S\rightarrow \S'$ of filtered semigroups is a bijective morphism whose inverse is also a morphism. We also say that two filtrations $R_1\supset R_2\supset\dots$ and $R'_1\supset R'_2\supset\dots$ on the same semigroup $\S$ are equivalent if the identity map is an isomorphism of filtered semigroups.
\end{definition}

\begin{lemma} \label{induce-semigroup-morph} %
  With the notation in Definition~\ref{def-morph-semi}, the morphism $\beta\colon\S\rightarrow \S'$ of filtered semigroups induces a natural homomorphism of semigroup $\beta_{ij}\colon S_j\rightarrow S'_i $ where $S_j=\S/R_j$ and $S'_i=\S'/R'_i$ such that the following diagram commutes.
 \begin{align*}
   \xymatrix@R+0em@C+0.7em{
     \mathcal{S} \ar[r]^{\beta}\ar[d]_{\pi_{j}} &\S'\ar[d]^{\pi'_i}\\
     S_j\ar[r]^{\beta_{ij}} & S'_i
   }
 \end{align*}
\end{lemma}

\begin{proof}
  The map $\beta_{ij}\colon S_j\rightarrow S_i$ defined naturally by $[m]\mapsto [\beta(m)]$ is well defined due to the condition $(\varphi\times\varphi)(R_j)\subset R'_i$. The rest of the proof is straightforward.
\end{proof}

We now define pro-affine semigroups that are the generalization of the affine semigroups that are the objects dual to classical affine toric varieties.

\begin{definition} \
  \begin{enumerate}
  \item A pro-affine semigroup $\mathcal{S}$ is a filtered semigroup with filtration $R_1\supset R_2\supset\dots$ of compatible equivalence relations in $\S$ that is complete, Hausdorff and such that every $\S/R_i$ is an affine semigroup.

  \item Let $\S$ be a filtered semigroup with filtration $R_1\supset R_2\supset\dots$. A filtered subsemigroup is a semigroup $\S'\subset \S$ endowed with the filtration of compatible equivalence relations $R_i\cap (\S'\times\S')$ on $\S'$.
  \end{enumerate}
\end{definition}

\begin{example}\label{100}\
  \begin{enumerate}
  \item We define the canonical filtration $\widetilde{R}_1\supset \widetilde{R}_2\supset\dots$ of equivalence relations on the semigroup $\Z^\omega$ by $\widetilde{R}_k=\{(m,m')\in \Z^\omega\times \Z^\omega\mid m_i=m'_i, \mbox{ for all } i\leq k\}$. By Proposition~\ref{proj-complete}, we conclude that $\Z^\omega$ is complete. Furthermore,  $\Z^\omega/R_i$ is naturally isomorphic $\Z^i$ with quotient morphism $\pi_i\colon \Z^\omega\rightarrow \Z^i$ the projection to the first $i$-th coordinates. Hence, $\Z^\omega/R_i$ is an affine semigroup and so the filtered semigroup $\Z^\omega$ is a pro-affine semigroup.

  \item The filtered subsemigroups $\Z_{\geq 0}^\omega$ of $\Z^\omega$ of arbitrary sequences of non-negative integers is also pro-affine with a similar argument as in (1).
    
  \item Any affine semigroup $S\subset\Z^i$ with the constant filtration given by the trivial equivalence relation is pro-affine.
    
  \item Let $e_i=(0,\dots ,0,1,0,\dots)\in \Z^\omega$, where the non-zero coefficient is located in the position $i>0$. The subsemigroup $\S=\Z_{\geq 0}^\omega\setminus\left\{e_1\right\}$ of $\Z^\omega$ is not complete and so is not pro-affine. Indeed, the sequence $\left\{a_i=e_1+e_i\right\}_{i>0}$ is Cauchy but not convergent in $\S$.
  \end{enumerate}
\end{example}

\begin{theorem} \label{pro-affine-embedding} %
  Let $\S$ be a pro-affine semigroup, then $\S$ is isomorphic to a filtered subsemigroup of $\Z^\omega$. Moreover, we can assume that $\S$ is embedded in $\M$ with $\Z\S=\M$, where  $\M\simeq\Z^\omega$ or $\M\simeq\Z^k$ for some $k>0$.
\end{theorem}

\begin{proof}
  Letting $R_1\supset R_2\supset \dots$ be the filtration of compatible equivalence relations in $\S$ we let $S_i=\S/R_i$ and $\varphi_i\colon S_{i+1}\rightarrow S_i$ be the homomorphisms given by the inclusions $R_i\supset R_{i+i}$. Hence, we have a commutative diagram
  \begin{equation*}
  \xymatrix@R+0em@C+0.7em{
 S_1 \ar[d]&\ar[l]_-  {}S_2 \ar[d] &\ar[l]_- {} S_3\ar[d]&\ar[l]_-{}\cdots\\
\Z S_1 &\ar[l]^- {} \Z S_2 & \ar[l]^-{} \Z S_3&\ar[l]^-{} \cdots
  }
\end{equation*}
where $\Z S_i$ is the group generated by $S_i$ for any embedding $S_i\hookrightarrow \Z^k$ and the homomorphisms $\Z S_{i+1}\rightarrow \Z S_{i}$ are induced by $S_{i+1}\rightarrow S_{i}$, for all $i>0$. Since the homomorphisms in the upper system are surjective, the same holds for the lower system. Hence, the lower projective system is split. If the homomorphisms in the lower system become also injective for $i$ large enough, then the projective limit of the lower system is isomorphic to $\Z^k$ for some $k\geq0$. Furthermore, since $\Z^k$ is embedded in $\Z^\omega$ the first statement follows in this case. Assume now that there is no integer $i>0$ such that the homomorphisms in the lower system become injective for all integer $j>i$. In this case, by Example~\ref{split-system} we have that the lower projective limit is isomorphic to $\Z^\omega$ and under this isomorphism we have that $\li S_i\subset \Z^\omega$ is an embedding of filtered semigroups. Since $\S$ is Hausdorff and complete, by Remark~\ref{limit-equal-complete} we have $\S=\li S_i$. The second statement follows directly from the construction above in this proof.
\end{proof}

In the following example we show the surprising consequence of the
Specker Theorem (Lemma~\ref{specker}) that every group homomorphism
$\beta\colon\Z^\omega\rightarrow \Z^\omega$ is a morphism of filtered
semigroups for the canonical filtration $\widetilde{R}_i$.

\begin{example} \label{homo-morphism-ex}
  \begin{enumerate}
  \item Every homomorphism $\beta\colon\Z^\omega\rightarrow \Z^\omega$ is a morphism of filtered semigroups with respect to the canonical filtration. Indeed, since $\Z^\omega$ is a group, we have that $E_{0,k}$ is a subgroup of $\Z^\omega$ and 
    $$\widetilde{R}_k=\bigcup_{m\in \Z^\omega}(m+E_{0,k})\times(m+E_{0,k})\,$$ %
    Hence, it is enough to show that for every $i>0$ there exists $j>0$ such that $\beta(E_{0,j})\subset E_{0,i}$. By Lemma~\ref{specker}, the composition $\pi_i\circ\beta\colon \Z^\omega\rightarrow \Z^i$ corresponds to an element in $(p_1,\dots,p_i)\in(\Z^\infty)^i$ under the isomorphism $\operatorname{Hom}(\Z^\omega,\Z)\simeq \Z^\infty$ given by the duality map, see also \cite[Theorem 94.3 and Corollary 94.5]{F73}. By definition of inductive limit, each $p_i\in \Z^{j_i}$ for some $j_i>0$. Taking $j$ to be the maximum of $\{j_1,\dots,j_i\}$ we obtain that $\beta(E_{0,j})\subset E_{0,i}$.

  \item A similar argument shows that every homomorphism $\beta\colon\Z^\omega\rightarrow \Z^k$ is a morphism of filtered semigroups with respect to the canonical filtration in $\Z^\omega$ and trivial filtration in $\Z^k$, for every $k\geq0$.
  \end{enumerate}
\end{example}

The above example allows us to prove that every homomorphism between pro-affine semigroups is a morphism.

\begin{proposition} \label{homo-morphism-prop} %
  Let $\S$ and $\S'$ be pro-affine semigroups. If $\beta\colon\S\rightarrow \S'$ is any homomorphism of semigroups then $\beta$ is a morphism of filtered semigroups.
\end{proposition}

\begin{proof}
  By Theorem~\ref{pro-affine-embedding}, we can assume that $\S$ is a subsemigroup of $\M=\Z^\omega$ or $\M=\Z^k$ for some $k\geq 0$ with $\Z S=\M$. Similarly, we can assume that $\S'$ is a subsemigroup of $\M'=\Z^\omega$ or $\M'=\Z^\ell$ for some $\ell\geq 0$ with $\Z \S'=\M'$. The homomorphism $\beta$ can be extended to a homomorphism $\widehat{\beta}\colon \M\rightarrow \M'$ via $m-m'\mapsto \beta(m)-\beta(m')$.  If $\M=\Z^k$, then $\widehat{\beta}$ is trivially a morphism of filtered semigroups since the filtration by equivalence relation on $\Z^k$ is trivial. Furthermore, if $\M=\Z^\omega$ the homomorphism $\widehat{\beta}$ is also a morphism of filtered semigroups by Example~\ref{homo-morphism-ex}. Now, the proposition follows since $\S$ and $\S'$ are filtered subsemigroups of $\M$ and $\M'$, respectively.
\end{proof}

\begin{remark} \label{unique-structure} %
  It follows from Proposition~\ref{homo-morphism-prop} above that two different filtrations $R_1\supset R_2\supset\dots$ and $R'_1\supset R'_2\supset\dots$ of compatible equivalence relations in a pro-affine semigroup $\S$ are always equivalent since the identity is an isomorphism of semigroups and so it is also an isomorphism of filtered semigroups. 
\end{remark}

It is straightforward to prove, mimicking the classical argument for metric spaces, that a subsemigroup in a complete filtered semigroup is complete if and only if it is closed. This allows us to derive the following corollary that acts as alternative definition of pro-affine semigroups. Recall that $\Z^\omega/\widetilde{R}_i$ is naturally isomorphic $\Z^i$ with quotient morphism $\pi_i\colon \Z^\omega\rightarrow \Z^i$ the projection to the first $i$-th coordinates.

\begin{corollary} \label{other-def-pro-affine} %
  An abstract semigroup $\S$ admits a filtration by compatible equivalence relations on $\S$ making $\S$ a pro-affine semigroup if and only if there exists an embedding $\iota\colon\S\hookrightarrow \Z^\omega$ where $\iota(\S)$ is closed and $(\pi_i\circ\iota)(\S)$ is finitely generated for every $i>0$. Moreover, if such a filtration exits, then it is unique (up to equivalence).
\end{corollary}

\begin{proof}
  If $\S$ admits a structure of pro-affine semigroup, then the corollary follows from Theorem~\ref{pro-affine-embedding}. On the other hand, if $\S$ is embedded in $\Z^\omega$, then it inherits a filtration $R_1\supset R_2\supset$ from this embedding. By definition $\S/R_i\simeq (\pi_i\circ\iota)(\S)$ which is assume to be finitely generated. Furthermore, $\S$ is complete with the induced filtration since $\iota(S)$ is closed in $\Z^\omega$ and $\S$ is Hausdorff since $\ZZ^\omega$ is. This yields that $\S$ is a pro-affine semigroup with this filtration. Finally, the uniqueness statement follows from Proposition~\ref{homo-morphism-prop} and Remark~\ref{unique-structure}.
\end{proof}

\section{Affine toric ind-varieties and pro-affine semigroups}
\label{sec:dual-category}

In this section we prove that the category of affine toric ind-varieties with toric morphisms is dual to the category of pro-affine semigroups with homomorphisms of semigroups.

Given an affine toric ind-variety $\mathcal{V}$ with toric filtration $V_1\hookrightarrow V_2\hookrightarrow\dots$, applying the functor $\S(\bigcdot)$ defined in Section~\ref{sec:toric-varieties}, we obtain a projective system
\begin{align*}\label{variety-to-semigroup}
  \xymatrix@R+0.1em@C+0.7em{
    \ar@{~>}[d] \ar[r]^{\varphi_1}V_1 &\ar@{~>}[d]\ar[r]^{\varphi_2} V_2 & \ar@{~>}[d]\ar[r]^{\varphi_3} V_3& \cdots \\
    S_1 &\ar[l]_{\S(\varphi_1)}S_2  &\ar[l]_{\S(\varphi_2)} S_3&\ar[l]_{\S(\varphi_3)}\cdots
  }
\end{align*}
where each semigroup $S_i=\S(V_i)$ is the affine semigroup associated to the toric variety $V_i$, i.e, $\C[V_i]=\C[S_i]$ and $\S(\varphi_i)\colon S_{i+1}\rightarrow S_i$ is the semigroup homomorphism corresponding to the toric morphism $\varphi_i\colon V_i\rightarrow V_{i+1}$ \cite[Proposition~1.3.14]{CLS11}. We define the semigroup $\S(\V)$ associated to $\mathcal{V}$ as the projective limit $\li S_i$ of this projective system. By Proposition~\ref{proj-complete} and the paragraph preceding it, we have that $\S(\V)$ is a pro-affine semigroup.

On the other hand, given a pro-affine semigroup $\S$ with the filtration $R_1\supset R_2\supset\dots$ of compatible equivalence relations on $\S$, we let $S_1\leftarrow S_2\leftarrow\dots $ be the associated projective system of semigroups where each $S_i=\S/R_i$ is an affine semigroup and the homomorphisms $\varphi_i\colon S_{i+1}\rightarrow S_i$ are given by $[m]_{i+1}\mapsto [m]_i$, where $[m]_i$ is the class of $m\in \S$ inside the quotient $S_i$. The homomorphisms $\varphi_i$ are surjetive. Hence, applying the functor $\V(\bigcdot)$ defined in Section~\ref{sec:toric-varieties} for toric varieties, we obtain an inductive system of closed embeddings
\begin{align*}
  \xymatrix@R+0.1em@C+0.7em{
  S_1 \ar@{~>}[d]&\ar[l]_{\varphi_1}S_2 \ar@{~>}[d] &\ar[l]_{\varphi_2} S_3\ar@{~>}[d]&\ar[l]_{\varphi_3}\cdots\\
  \ar[r]^{\V(\varphi_1)}V_1 &\ar[r]^{\V(\varphi_2)} V_2 & \ar[r]^-{\V(\varphi_3)} V_3& \cdots}
\end{align*}
where each $V_i=\V(S_i)$ is the toric variety associated to the semigroup $S_i$ and $\V(\varphi_i)\colon V_i\rightarrow V_{i+1}$ is the toric morphism corresponding to the semigroup homomorphism $\varphi_i\colon S_{i+1}\rightarrow S_i$. The corresponding inductive limit $\ld V_i$ of this system is an affine toric ind-variety by Theorem~\ref{toric-filtration} that we denote by $\V(\S)$. The  ind-torus acting on $\V(\S)$ is $\T=\ld T_i$, where $T_i$ is the algebraic torus acting on $V_i$. It is clear that these constructions provide a bijection between affine toric varieties and pro-affine semigroups up to isomorphisms.

Let now $\V$ be an affine toric ind-variety and let $\S=\S(\V)$. In general, projective limits do not commute with direct sums, hence we cannot expect to have, as in the classical case, an isomorphism between the ring of regular functions $\C[\mathcal{V}]$ on $\V$ and the semigroup algebra $\C[\mathcal{S}]$, see Example \ref{ex-referee} below. Nevertheless, the semigroup algebra carries a natural descending filtration of ideals $ I_1\supset I_2\supset \dots$, where $I_{i}=\ker \pi_{i}$ and $\pi_i$ is the natural projection $\pi_i\colon\C[\S]\rightarrow \C[V_i]$, for all $i>0$ induced by the projections $\overline{\pi}_i\colon\S\rightarrow S_i$ coming from the projective limit. It follows directly from \cite[Chapter~9, Theorem~10]{N68} that the algebra $\C[\V]$ is the completion of $\C[\S]$ with respect to $I_1\supset I_2\supset \dots$.

In the following proposition, we summarize the considerations above.

\begin{proposition} \label{equivalence-object} %
  The assignments $\V\mapsto S(\V)$ for every affine toric ind-variety and $S\mapsto \V(\S)$ for every pro-affine semigroup are inverses up to isomorphism, i.e., $\V(\S(\V))$ is isomorphic to $\V$ for every affine toric ind-variety and $\S(\V(\S))$ is isomorphic to $\S$ for every pro-affine semigroup $\S$. Furthermore, for every affine toric ind-variety $\V$, the ring of regular functions $\C[\V]$ is isomorphic as filtered algebra to the completion of $\C[\S]$.
\end{proposition}

We will also need the following lemma generalizing the usual equivalent statement in the classical case.

\begin{lemma} \label{S-span-M} %
  Let $\V$ be an affine toric ind-variety with acting ind-torus $\T$ whose character lattice is $\M$. Then $\S(\V)$ is naturally embedded in $\M$ with $\Z\S(\V)=\M$. On the other hand, let $\S$ be a pro-affine semigroup embedded in $\M\simeq\Z^\omega$ or $\M\simeq\Z^k$ for some $k\geq 0$ as filtered semigroup with $\Z \S=\M$. Then the character lattice of the ind-torus $\T$ acting on the affine toric ind-variety $\V(\S)$ is naturally isomorphic to $\M$.
\end{lemma}

\begin{proof}
  The case where $\M\simeq \Z^k$ corresponds to the classical case of affine toric varieties. Hence, we will only deal with the case where $\M\simeq \Z^\omega$. Assume first that $\V$ is an affine toric ind-variety. With the above notation, by the classical case we have that each $S_i$ is naturally embedded in the character lattice $M_i$ of the algebraic torus $T_i$ acting on $V_i$ with $M_i=\Z S_i$. By Theorem~\ref{toric-filtration}, we have that $\T$ equals the inductive limit $\ld T_i$. Furthermore, by Proposition~\ref{ind-torus-characters} we have that $\M$ equals $\li M_i$. The first assertion now follows. On the other hand, given $\S$ embedded in $\M\simeq \Z^\omega$, we let $M_i$ be the character lattice of the torus $T_i$ acting on $V_i$. By the classical finite dimensional case of the lemma, we have $\Z S_i=M_i$. The result now follows again from Proposition~\ref{ind-torus-characters}.
\end{proof}

We come now to morphisms in both categories. Let first $\S$ and $\S'$ be pro-affine semigroups and let $\beta\colon\S\rightarrow \S'$ be a semigroup homomorphism. By Proposition~\ref{homo-morphism-prop} the pro-affine semigroups $\S$ and $\S'$ admit filtrations of equivalence relations $R_1\supset R_2\supset\dots$ and $R'_1\supset R'_2\supset\dots$, respectively, such that $\beta$ is a morphism of filtered semigroups with respect to these filtrations. We let $\V=\V(\S)$ and $\V'=\V(\S')$ be the corresponding affine toric ind-varieties defined above with the toric filtrations $V_1\hookrightarrow V_2\hookrightarrow\dots$ and $V'_1\hookrightarrow V'_2\hookrightarrow\dots$, respectively, where $V_i=\V(S_i)$, $V'_i=\V(S'_i)$ and the closed embeddings are $\V(\varphi_i)$ and $\V(\varphi'_i)$, respectively. We define a homorphism $\C[\S]\rightarrow\C[\S']$ of semigroup algebras by $\chi^m\mapsto \chi^{\beta(m)}$, for all $m\in \S$. By abuse of notation, we denote this map also by $\beta\colon \C[\S]\rightarrow\C[\S']$.

\begin{lemma}
  The homomorphism $\beta\colon \C[\S]\rightarrow\C[\S']$ is a continuous homomorphism of topological algebras and so we can extend $\beta$ to an unique continuous homomorphism $\V(\beta)^*\colon \C[\V]\rightarrow\C[\V']$ whose comorphism defines a toric morphism of affine toric ind-varieties $\V(\beta)\colon \V'\rightarrow \V$.
\end{lemma}

\begin{proof}
  To prove that $\beta\colon \C[\S]\rightarrow\C[\S']$ is continuous we have to prove that for all $i>0$ there exists $j>0$ such that $\beta(I_j)\subset I'_i$. Here $I_{j}=\ker \pi_{j}$ and $\pi_j$ is the projection $\pi_j\colon\C[\S]\rightarrow \C[S_j]$ induced by $\S\rightarrow S_j$, for all $j>0$ and similarly $I'_{i}=\ker \pi'_{i}$ and $\pi'_i$ is the projection $\pi'_i\colon\C[\S']\rightarrow \C[S'_i]$ induced by $\S'\rightarrow S'_i$, for all $i>0$.

  Let $i>0$ be an integer. By the definition of morphism of filtered semigroup, there exists $j>0$ such that $(\beta\times\beta)(R_j)\subset R'_i$. Let $f=\sum a_m\chi^m$ be an element in $I_j$ where the sum is finite. Belonging to $I_j$ is equivalent to $\pi_j(f)=\sum a_m\chi^{\pi_{j}(m)}=0$. On the other hand, $\pi'_{i}(\beta(f))=\sum a_m\chi^{(\pi'_{i}\circ\beta)(m)}$. By Lemma~\ref{induce-semigroup-morph}, the homomorphism $\beta$ induces a homomorphism $\beta_{ij}\colon S_j\rightarrow S'_i$ and we have $\pi'_{i}\circ\beta=\beta_{ij}\circ \pi_{j}$ so we have $\pi'_{i}(\beta(f))=\sum a_m\chi^{(\beta_{ij}\circ \pi_{j})(m)}=\beta_{ij}\big(\sum a_m\chi^{\pi_{j}(m)}\big)=0$. We conclude that $\beta(I_j)\subset I'_i$ and so $\beta\colon \C[\S]\rightarrow\C[\S']$ is continuous.

  Finally, the algebra $\C[\S]$ is dense in $\C[\V]$ by the second statement of Proposition~\ref{equivalence-object}. Hence, the homomorphism $\beta$ can be extended to a continuous homomorphism $\V(\beta)^*\colon \C[\V]\rightarrow\C[\V']$ as required, see \cite[Ch.9, Th. 5]{N68}.  Moreover, by Proposition~\ref{toric-morphism-filtration}, the morphism $\V(\beta)\colon \V'\rightarrow \V$ is toric.
\end{proof}

Let now $\alpha\colon\V\rightarrow \V'$ be a toric morphism of affine toric ind-varieties and let $\S=\S(\V)$ and $\S'=\S(\V')$ be the corresponding pro-affine semigroups. By Lemma~\ref{S-span-M} we have that $\S$ and $\S'$ are naturally embedded in $\M$ and $\M'$, respectively. In particular, we have that $\alpha|_{\T_\V}\colon \T_\V\rightarrow \T_{\V'}$ is a homomorphism of ind-groups and so by Lemma~\ref{morph-tori} the comorphism $(\alpha|_{\T_{\V}})^*$ induces a semigroup homomorphism $\alpha^\vee\colon \M'\rightarrow \M$ on the character lattices via $(\alpha|_{\T_{\V}})^*(\chi^m)=\chi^{\alpha^\vee(m)}$. Furthermore, given $m\in \S'$, the regular function $\chi^m\in \C[\V']$ is mapped to the regular function $\chi^{\alpha^\vee(m)}\in \C[\V]$. This yields $\alpha^\vee(m)\in \S$, for all $m\in \S'$. Hence $\alpha^\vee$ restricts to a homomorphism  $\S'\rightarrow \S$. We denote this homomorphism by $\S(\alpha)$.

In the following proposition, we summarize the considerations above.

\begin{proposition} \label{equivalence-morphism} %
  let $\S$ and $\S'$ be pro-affine semigroups. Then, for every homomorphism $\beta\colon\S\rightarrow \S'$ the map  $\V(\beta)\colon\V(\S')\rightarrow \V(\S)$ is a toric morphism of affine toric ind-varieties. Moreover, for every toric morphism $\alpha\colon\V(\S')\rightarrow \V(\S)$ there exists a unique homomorphism $\beta\colon \S\rightarrow \S'$ such that $\alpha=\V(\beta)$. In particular, for every pair of pro-affine semigroups $\S$ and $\S'$ there is a bijection between semigroup homomorphisms $\S\rightarrow \S'$ and toric morphisms $\V(S')\rightarrow \V(\S)$.
\end{proposition}

The assignment $\V(\bigcdot)$ is a contravariant functor, i.e., $\V(\operatorname{id})=\operatorname{id}$ and $\V(\beta'\circ\beta)= \V(\beta)\circ \V(\beta')$, for every pair of semigroup homomorphisms $\beta\colon\S\rightarrow\S'$ and $\beta'\colon\S'\rightarrow\S''$, where $\S$, $\S'$ and $\S''$ are pro-affine semigroups. This follows directly from the definition of $\V(\beta)$ as the comorphism of the unique extension of the morphism $\C[\S]\rightarrow \C[\S']$ given by $\chi^m\mapsto \chi^{\beta(m)}$.

On the other hand, the assignment $\S(\bigcdot)$ is also a contravariant functor. Indeed, let $\alpha'\colon\V''\rightarrow \V'$ and $\alpha\colon\V'\rightarrow \V$ be morphisms of affine toric ind-varieties $\V$, $\V'$ and $\V''$. By Proposition~\ref{equivalence-object} and Proposition~\ref{equivalence-morphism}, there exist pro-affine semigroups $\S$, $\S'$, $\S''$ such that $\V=\V(\S)$, $\V'=\V(\S')$ and $\V''=\V(\S'')$ with morphisms $\beta\colon \S\rightarrow \S'$ and $\beta'\colon\S'\rightarrow \S''$ such that $\beta=\S(\alpha)$ and $\beta'=\S(\alpha')$. By Proposition~\ref{equivalence-morphism}, we have $\V(\beta'\circ\beta)=\alpha\circ\alpha'$ or, equivalently,  $\beta'\circ\beta=\S(\alpha\circ\alpha')$ so that $\S(\alpha')\circ\S(\alpha)=\S(\alpha\circ\alpha')$.

In the following theorem, that is our main result, we summarize the results in this section.

\begin{theorem} \label{main-theorem} \
  \begin{enumerate}
  \item The assignment $\V(\bigcdot)$ is a contravariant functor from the category of pro-affine semigroups with homomorphisms of semigroups to the category of affine toric ind-varieties with toric morphisms.
  \item The assignment $\S(\bigcdot)$ is a contravariant functor from the category of affine toric ind-varieties  with toric morphisms to the category of pro-affine semigroups  with homomorphisms of semigroups.
  \item The pair $(\V(\bigcdot),\S(\bigcdot))$ is a duality between the categories of affine toric ind-varieties and pro-affine semigroups.
  \end{enumerate}
\end{theorem}

A well-known feature of the classical duality between affine toric varieties and affine semigroups is the correspondence between points on the toric variety and semigroup homomorphism to $(\C,\cdot)$. In the following proposition, we generalize this result to the case of affine toric ind-varieties. 

Recall that  a semigroup $\S$ has the cancellation property if $m+m'=m+m''$ implies $m'=m''$, with $m,m',m''\in \S$. Let $(\C,\cdot)$ be the semigroup of complex numbers under multiplication. This semigroup is not pro-affine since it does not have the cancellation property and all pro-affine semigroup inherit the cancellation property from the embedding in $\Z^\omega$ shown in Corollary~\ref{other-def-pro-affine}. 

We endow $(\C,\cdot)$ with the trivial descending filtration $R'_1\supset R'_2\supset\dots$ of compatible equivalence relations $R'_i=\{(t,t)\in \C\times \C\mid t\in\C\}$ so that $\C/R_i\simeq \C$. Unlike the case of pro-affine semigroups, not every semigroup homomorphism $\S\rightarrow (\C,\cdot)$ is a filtered morphism. See \cite[page 159]{F73} and apply the fact that $(\C,\cdot)$ contains an isomorphic copy $Q$ of the additive group of the rational numbers. For instance, we can take $Q=\{\exp(q)\mid\ q\in\QQ \}$ where $\exp\colon \CC\rightarrow \CC^*$ is the usual exponential map.

\begin{proposition}
  Let $\mathcal{V}$ be an affine toric ind-variety and let $\S=\S(\V)$. Then there are bijective correspondence between the following:
  \begin{enumerate}
  \item Points $v$ in $\mathcal{V}$.
  \item Closed maximal ideals $\mathfrak{m}$ in $\C[\mathcal{V}]$ that is equal to the completion of $\C[\S]$.
  \item Morphisms of filtered semigroups $\Lambda\colon\mathcal{S}\rightarrow (\C,\cdot)$.
  \end{enumerate}
\end{proposition}

\begin{proof}
  The equivalence of $(1)$ and $(2)$ is general for ind-varieties and was first proven in \cite{Kam03}. Let $R_1\supset R_2\supset\dots$ be the filtration of compatible equivalence relation in $\S$ and let $\Lambda\colon\mathcal{S}\rightarrow \C$ be a filtered semigroup morphism. By the definition of filtered semigroups, there exists $j>0$ such that $(\Lambda\times\Lambda)(R_j)$ is contained in the diagonal in $\C\times\C$ defining the trivial equivalence relation in $\C$. By Lemma~\ref{induce-semigroup-morph}, the morphism $\Lambda$ induces a semigroup homomorphisms $\Lambda_j\colon S_j\rightarrow \C$, where $S_j=\S/R_j$. The homomorphism $\Lambda_j\colon S_j\rightarrow \C$ induces a surjective $\C$-algebra homomorphism $\overline{\Lambda}_j\colon\C[S_j]\rightarrow \C$ given by $\chi^m\mapsto \Lambda_j(m)$. Since $\C$ is a field, we have $\overline{\mathfrak{m}}=\ker \overline{\Lambda}_j$ is a maximal ideal. The preimage $\mathfrak{m}$ of $\overline{\mathfrak{m}}$ by the homomorphism $\widehat{\pi}_j\colon\C[\V]\rightarrow \C[S_j]$ coming from the projective system $\C[S_1]\leftarrow\C[S_2]\leftarrow\dots$ is also maximal. By \cite[Proposition 1.2.2]{Kam03} we have that $\mathfrak{m}$ is closed since $\widehat{I}_j=\ker\widehat{\pi}_j$ is subset of $\mathfrak{m}$.

  On the other hand, let $\mathfrak{m}$ be a closed maximal ideal in $\C[\V]$. By \cite[Proposition 1.2.2]{Kam03}, there exist $j>0$ and $\overline{\mathfrak{m}}$ a maximal ideal of $\C[S_j]$ such that $\mathfrak{m}$ is the preimage of $\overline{\mathfrak{m}}$ by $\widehat{\pi}_j$. This maximal ideal $\overline{\mathfrak{m}}$ defines an algebra homomorphisms $\overline{\Lambda}_j\colon\C[S_j]\rightarrow \C\simeq \C[S_j]/\overline{\mathfrak{m}}$.  By \cite[proposition 1.3.1]{CLS11}, this algebra homomorphism defines a semigroup homomorphism $\Lambda_j\colon S_j\rightarrow \C$ given by $\Lambda_j(m)= \overline{\Lambda}_j(\chi^m)$.  We define $\Lambda\colon \S\rightarrow \C$ by $\Lambda=\Lambda_j\circ \pi_j$, where $\pi_j\colon \S\rightarrow S_j$ is the quotient morphism. The semigroup homomorphism $\Lambda$ is a filtered semigroup morphism since $(\Lambda \times \Lambda)(R_j)$ is contained in the diagonal in $\C\times\C$ defining the trivial equivalence relation in $\C$. It is a straightforward verification that both this constructions provide the required bijection.
\end{proof}

\section{Examples} \label{sec:examples}

To conclude the paper, we provide the following three examples of affine toric ind-varieties.

\begin{example}
The ind-torus $\T=(\CC^*)^\infty$ is a toric ind-variety. Furthermore, since the algebra of regular functions of $(\CC^*)^i$ is $\CC[\ZZ^i]$ we obtain that $\S(\T)=\ZZ^\omega$ by Example~\ref{split-system}.
\end{example}

\begin{example} \label{ex-referee}
The infinite dimensional affine space  $\CC_1^\infty\simeq \CC^\infty$ defined in Example~\ref{15} is a toric ind-variety. Furthermore, since the algebra of regular functions of $\CC^i$ is $\CC[\ZZ_{\geq 0}^i]$ we obtain that $\S(\CC_1^\infty)=\ZZ_{\geq 0}^\omega$, see also Example~\ref{100}.

We take advantage of this example to show that in general $\C[\S]\subsetneq\C[\V]$. To do so, we show that $\CC[\ZZ_{\geq 0}^\omega]$ is not a complete topological ring. Recall that 
$$\CC[\ZZ_{\geq 0}^\omega]=\bigoplus_{m\in \ZZ_{\geq 0}^\omega}\CC\chi^m\,.$$
We also let 
$$x_i=\chi^m, \mbox{ with } m=(\underbrace{1,1,\ldots,1}_{i\mbox{-times}},0,0,\ldots), \quad f_1=x_1, \quad\mbox{and}\quad f_i=\frac{x_i}{2^{i-1}}+\sum_{k=1}^{i-1} \frac{x_k}{2^{k}}\,.$$
The sequence $(f_i)_{i>0}$ is Cauchy since
$$f_j-f_i=\frac{x_j}{2^{j-1}}+\sum_{k=i}^{j-1} \frac{x_k}{2^{k}}-\frac{x_i}{2^{i-1}} \quad \mbox{for all } j>i\,,$$
and a straightforward computation shows that 
$$\pi_{\ell}(f_j-f_i)=\left(\frac{1}{2^{j-1}}+\sum_{k=i}^{j-1} \frac{1}{2^{k}}-\frac{1}{2^{i-1}}\right)\chi^{m'}=0 \quad \mbox{for all } j>i\mbox{ and }  i,j > \ell\,.$$
Here $m'=(1,1,\ldots,1)\in \ZZ^\ell$ and  $\pi_\ell\colon \CC[\ZZ_{\geq 0}^\omega]\rightarrow \CC[\ZZ_{\geq 0}^\ell]$ is the natural morphism coming from the projective limit that in this case corresponds to the semigroup homomorphism $\ZZ_{\geq 0}^\omega\rightarrow \ZZ_{\geq 0}^\ell$ restriction to the $\ell$-th first coordinates.
Finally, the sequence  is not convergent in $\CC[\ZZ_{\geq 0}^\omega]$ since the limit is an infinite sum that cannot belong to the direct sum in the definition of $\CC[\ZZ_{\geq 0}^\omega]$.
\end{example}

\begin{example}
Let $\V$ be the affine toric ind-variety $V_1\hookrightarrow V_2\hookrightarrow\ldots$, where  $V_i$ is the $i$-dimensional affine toric variety given in $\CC^{i+1}$ with coordinates $(y,x_1,\ldots,x_i)$ by the equation $y^2=x_1\cdots x_i$, where $V_i\hookrightarrow V_{i+1}$ is the closed embedding given by the toric map $(y,x_1,\ldots,x_i)\mapsto (y,x_1,\ldots,x_i,1)$. 

From the equation $y^2=x_1\cdots x_i$ we obtain that the embedding of the acting torus of $V_i$ in the acting torus of $\CC^{i+1}$ corresponds to the homomorphism of the character lattices $\ZZ^{i+1}=M(\CC^{i+1})\rightarrow M(V_i)=\ZZ^{i}$ 
 given by the matrix
$$
\begin{pmatrix}
1 & 2  & 0 & 0 &  \ldots & 0 \\
0 & -1 & 1 & 0 &  \ldots & 0 \\
0 & -1 & 0 & 1 &   & 0 \\
\vdots & \vdots & \vdots & \vdots &  \ddots & \vdots \\
0 & -1 & 0 & 0 &  \ldots & 1 
\end{pmatrix}\,.
$$
Hence, we obtain that $V_i=\spec \CC[S_i]$ where $S_i=\omega_i\cap \ZZ^i$ and $\omega_i$ is the cone spanned in $\RR^i$ by the rays
$$
\begin{pmatrix}
2 \\ -1 \\ -1 \\ \vdots \\ -1  
\end{pmatrix},
\begin{pmatrix}
0 \\ 1 \\ 0 \\ \vdots \\ 0  
\end{pmatrix},
\begin{pmatrix}
0 \\ 0 \\ 1 \\ \vdots \\ 0  
\end{pmatrix},\ldots, \begin{pmatrix}
0 \\ 0 \\ 0 \\ \vdots \\ 1  
\end{pmatrix}\,.
$$
The semigroup $S_i$ is thus given in $\ZZ^i$ by 
$$S_i=\big\{(m_1,\ldots, m_i)\in \ZZ^i \mid m_1\geq 0, \mbox{ and } m_1+2m_j\geq 0, \mbox{ for all }  j=2,3,\ldots,i\big\}\,. $$
Furthermore, the closed embedding $V_i\hookrightarrow V_{i+1}$ corresponds to the map $\ZZ^{i+1}\rightarrow \ZZ^i$ of the respective character lattices given by the matrix
$$
\begin{pmatrix}
1 & 0 & 0 &  \ldots & 0 & 0 \\
0 & 1 & 0 &  \ldots & 0 & 0 \\
0 & 0 & 1 &   & 0 \\
\vdots & \vdots & \vdots &  \ddots & \vdots & \vdots \\
0 & 0 & 0 &  \ldots & 1 & 0 
\end{pmatrix}\,.
$$
Taking the projective limit we obtain that the pro-affine semigroup $\S$ corresponding to the affine toric ind-variety $\V$ is given by
$$\S=\big\{(m_1,m_2,\ldots)\in \ZZ^\omega \mid m_1\geq 0, \mbox{ and } m_1+2m_j\geq 0, \mbox{ for all }  j\geq 2\big\}\,. $$
\end{example}

\end{document}